%% file: Dividendenproblem.tex
\documentclass[12pt]{article}
\usepackage[latin1]{inputenc}
\usepackage{style}
\usepackage{geometry}
\usepackage{graphicx}
\usepackage{amsthm}
\usepackage{amssymb}
\usepackage{xcolor} 
\usepackage{enumitem}
\usepackage{accents}
\newcommand{\ubar}[1]{\underaccent{\bar}{#1}} 
\usepackage{etoolbox} 

\theoremstyle{plain}
\newtheorem{theorem}{Theorem}[section]
\newtheorem{lemma}[theorem]{Lemma}

\newtheorem{prop}[theorem]{Proposition}

\theoremstyle{definition}
\newtheorem{rem}[theorem]{Remark}
\newtheorem{definition}[theorem]{Definition}

\AfterEndEnvironment{theorem}{\noindent\ignorespaces}
\AfterEndEnvironment{definition}{\noindent\ignorespaces}
\AfterEndEnvironment{proof}{\noindent\ignorespaces}
\AfterEndEnvironment{rem}{\noindent\ignorespaces}

\makeatletter
\numberwithin{equation}{section} 

\def\endalpheq{\setcounter{equation}{\value{tempcou}}\global\@ignoretrue}
\makeatother

\usepackage[onehalfspacing]{setspace} 

\newcommand{\onebr}[1]{\one_{\{#1\}}} 

\title{Stochastic Control of Dividends with a Drawdown Penalty}
\author{
        Kira Dudziak\footnote{University of Cologne, Department of Mathematics and Computer Science, Division of Mathematics, Cologne, Germany. E-mail: kira.dudziak@uni-koeln.de} \and Hanspeter Schmidli\footnote{University of Cologne, Department of Mathematics and Computer Science, Division of Mathematics, Cologne, Germany. E-mail: hanspeter.schmidli@uni-koeln.de}
}

\begin{document}

\maketitle

\begin{abstract}
\noindent
We consider a diffusion risk model where dividends are paid at rate $U(t)\in[0,u_0]$. We are interested in maximising the dividend payments under a drawdown constraint, that is, we penalise a drawdown size larger than a level $d>0$. We show that the optimal dividend rate $U(t)$ is either zero or the maximal rate $u_0$ and determine the optimal strategy. Moreover, we derive an explicit expression for the value function by solving a system of differential equations.
\end{abstract}
\medskip


\noindent {Classification:} Primary 91B05; Secondary 91G05, 93E20, 60J60
\medskip

\input{Introduction}
\input{Propertiesofv}
\input{ADirectApproach}
\input{ConstructionofSolution}
\input{NumericalExamples}
\input{Conclusion}

\bibliographystyle{abbrv}
\bibliography{Literatur}
\end{document}

%% file: Introduction.tex
\section{Introduction} \label{section:introduction}
\subsection{Problem Formulation}
The distribution of dividends is considered an indicator of a company's
financial strength. Since dividends also function as returns, shareholders
will prefer higher dividend payments. Beyond that, an insurance company has a
special responsibilty towards its policyholders: In the event of a claim, they
rely on the insurer to meet his financial obligations. Financial stability is
therefore another important goal, in particular, requested by the regulator. In order to measure the stability of the surplus process, an object of interest is the size of the so-called \textit{drawdown}, that is the distance between the current capital and the last historical maximum. The challenge is to reconcile these two objectives, as dividend payments reduce capital and could lead to a larger drawdown.

\noindent
In the literature optimal dividends have been extensively studied. The
problem has been originally treated by De Finetti \cite{DeFinetti} and Gerber
\cite{Gerber69}. For an overview see also \cite{AzcMul} or
\cite{Schmidli}. The corresponding optimal strategies lead always to ruin
which could lead to problems with the supervisor. The goal of the insurer has
therefore also to be financial stability. A possibility to measure stability
could be to measure how much time the drawdown spends in a critical area. This
problem was considered by \cite{LVB} and \cite{Brinker2022}, see also
the references in the latter publications. If drawdowns are considered as only
measure, the insurer would never make profits. Alternatively one could in
addition reward an increase of the maximum, which also could be interpreted
as dividend payments according to a barrier dividend strategy. However, a
barrier dividend strategy will be hardly seen in the real world. In this paper
we thus allow more flexible strategies but stick to continuous dividend
payments at a bounded rate. 

We work on a complete probability space $(\Omega,?F,!!P)$ containing all the stochastic objects defined below. We assume that the surplus process is given by the diffusion approximation $x+\mu t + \sigma W(t)$, $t\geq 0$, with drift parameter $\mu$ and volatility parameter $\sigma$ and a standard Brownian motion $(W(t))_{t\geq 0}$. Let $?U$ be the set of adapted processes $(U(t))_{t\geq 0}$ with $0\leq U(t) \leq u_0$ and maximal dividend rate $u_0>0$. For a dividend strategy $U\in?U$ we consider the post-dividend surplus process $(X^U(t))_{t\geq 0}$ given by
\[
X^U(t) = x + \mu t + \sigma W(t) - \int_0^t U(s) \id s\;.
\]
Moreover, we define the running maximum of the process $(M_z^U(t))_{t\geq 0}$ by
\[
M_z^U(t) = \max\{z,\sup_{s\in[0,t]} X^U(s)\}
\]
and the drawdown process $(\Delta_z^U(t))_{t\geq 0}$ by
\[
\Delta_z^U(t) = M_z^U(t)- X^U(t)\;.
\]
Thus, we allow a past maximum $z$ at time zero. Note that the drawdown size is independent of the initial capital $x$. We therefore assume $x=0$. Our aim is now to maximise dividend payments while minimising the time in which the drawdown size exceeds a critical level $d>0$. This idea is represented by the \textit{value function}
\begin{equation}
v(z) = \sup_{U\in?U} v^U(z) = \sup_{U\in?U} \E\bigl[\beta\int_0^\infty \e^{-r t}U(t) \id t - \int_0^\infty \e^{-r t}\onebr{\Delta_z^U(t)>d} \id t\bigr]\;, \quad z\geq 0\;, \label{valuefunction}
\end{equation}
where the size of the weight parameter $\beta>0$ indicates the significance of dividend payments relative to the drawdown penalty. We include a preference factor $r>0$ such that present dividend payments $U(t)$ and critical drawdowns $\Delta(t)>d$ have a higher weight than those far in the future.
\subsection{First Results and HJB equation}
We have the following basic properties of the value function.
\begin{lemma} \label{lemma:propertiesofv}
The value function $v$ is decreasing and bounded by $(\beta u_0-1)/r\leq v(z) \leq \beta u_0/r$ and converges to its lower bound as $z\to\infty$.
\end{lemma}
\begin{proof}
Let $U\in?U$ be an admissible dividend strategy. Since
$\Delta_y^{U}(t)\leq \Delta_z^{U}(t)$ for all $y\leq z$ and $t\geq 0$,
it follows that $v^{U}(z)$ is a decreasing function. This yields $v^{U}(z)\leq
v^{U}(y) \leq v(y)$. By taking the supremum over all admissible strategies
$U\in?U$ on the left side, we find $v(z)\leq v(y)$.
It is clear that $v(z)\leq \beta u_0/r$ by not penalizing the drawdowns. For the lower bound we note that 
\begin{equation*}
v(z)\geq v^{u_0}(z) = \E\bigl[\beta\int_0^\infty \e^{-r t}u_0 \id t - \int_0^\infty \e^{-r t}\onebr{\Delta_z^{u_0}(t)>d} \id t\bigr] 
\geq \frac{\beta u_0-1}{r}\;.
\end{equation*}
We will derive the limit $\lim_{z\to\infty} v(z) = (\beta u_0-1)/r$ directly
from the construction in Theorem~\ref{thm:drawdownsensitive} and \eqref{gupper}.
\end{proof}
We will later see that we can find an expression for the value
function and an optimal dividend strategy. The key tool for tackling this
problem will be the following \textit{Hamilton-Jacobi-Bellman (HJB) equation}:
\begin{equation}
\sup_{u\in[0,u_0]} ?A^{u}(h)(z) := -r h(z) -\mu h'(z) + \frac{\sigma^2}{2}h''(z) + \sup_{u\in[0,u_0]}\{(\beta+h'(z))u\} = \one_{\{z>d\}}\;. \label{HJB}
\end{equation}
Before we give a verification theorem, we have to define what we understand under a solution to \eqref{HJB}. 
\begin{definition}\label{def:solution}
A function $h:[0,\infty)\to[(\beta u_0-1)/r,\beta u_0/r]$ is called a solution to \eqref{HJB} if the following conditions are fulfilled:
\begin{enumerate}
\item $h$ is continuously differentiable on $[0,\infty)$, \label{def:sol:contdiff}
\item $h$ is twice continuously differentiable on $[0,d]$ and $(d,\infty)$, respectively, and 
\item $\sup_{u\in[0,u_0]} ?A^{u}(h)(z) = \one_{\{z>d\}}$ for all $z\geq d$. \label{def:sol:HJB}
\end{enumerate}
\end{definition}
\begin{rem}
Note that $h'(0)$ and $h''(0)$ refer to the first two derivatives from the
right and $h''(d)$ refers to the second derivative from the left. Since
$\one_{\{z>d\}}$ is discontinuous at $z=d$, the same must also apply to the
second derivative of $h$.
\hfill$\blacksquare$
\end{rem}
Now we can state the Verification Theorem.
\begin{theorem}[Verification Theorem] \label{verification}
Let $h:[0,\infty)\to[(\beta u_0-1)/r,\beta u_0/r]$ be a decreasing solution to
\eqref{HJB} in the sense of Definition {\rm \ref{def:solution}}. Then $h(z)\geq v^U(z)$ for all $z\in[0,\infty)$ and every strategy $U\in?U$. Moreover, if
$h'(0)=0$ and if $$u^{*}(z)=
u_0\one_{\{h'(z) \geq -\beta\}}$$ defines an admissible strategy via $U^{*}(t)=u^{*}(\Delta_z^{U^{*}}(t))$, $t\geq0$, then
$h(z)=v(z)=v^{U^{*}}(z)=\sup_{U\in?U}v^{U}(z)$ for all $z\geq 0$.
\end{theorem}
\begin{proof}
We first claim that $h'(z)$ is bounded. We thus only have to consider the case
$h'(z) < -\beta$ where the supremum is taken for $u=0$. For such a $z$,
\[ \frac{\sigma^2}2 h''(z) = \one_{\{z>d\}} + r h(z) + \mu h'(z)\;.\]
This is strictly negative if $h'(z) < -\mu^{-1}(1 + \beta u_0)$. In the
latter case, $h'(z)$ is strictly decreasing, contradicting the assumption that $h(z)$ is
bounded. Thus, $h'(z)$ has to be bounded. In particular, the stochastic
integral below will be a martingale. Due to the discontinuity of $h''$ in
$z=d$, we have to use a generalised version of It\^o's formula: By applying
Theorem 2.1 of \cite{Peskir} for $F(t,z)=\e^{-rt}h(z)$ and $b(t)= d$, $t,z\geq
0$, to the continuous semimartingale $((t,\Delta_z^{U}(t)))_{t\geq 0}$, we find
\begin{eqnarray}
\e^{-rt}h(\Delta_z^{U}(t))&=&h(z)+\int_0^t\e^{-rs}h'(0) \id M_z^{U}(s)-\sigma\int_0^t\e^{-rs}h'(\Delta_z^{U}(s)) \id W(s) \nonumber\\
&\hspace{1cm}& + \int_0^t \e^{-rs}[ -rh(\Delta_z^{U}(s))-\mu
h'(\Delta_z^{U}(s))+U(s)h'(\Delta_z^{U}(s))+\frac{\sigma^2}{2}h''(\Delta_z^{U}(s))
\one_{\{\Delta_z^{U}(s)\neq d\}}] \id s \nonumber\\
&\leq& h(z)-\sigma\int_0^t\e^{-rs}h'(\Delta_z^{U}(s)) \id W_s +  \int_0^t \e^{-rs}[ \one_{\{\Delta_z^{U}(s)>d\}}-\beta U(s)] \id s\;, \label{eq:veri}
\end{eqnarray}
where we used $h'(0)\leq 0$ and the fact that $\{0\leq t<\infty:\Delta_z^U(t) = d\}$ has
$!!P$-almost surely Lebesgue measure zero. The latter can be seen by
\begin{eqnarray*}
\E[\lambda(\{0\leq t <\infty: \Delta_z^U(t) = d\})] = \E\bigl[\int_0^\infty \one_{\{\Delta_z^U(t) = d\}} \id t  \bigr] = \int_0^\infty !!P[\Delta_z^U(t)=d] \id t = 0
\end{eqnarray*}
which implies $\lambda(\{0\leq t<\infty:\Delta_z^U(t) = d\})=0$ $!!P$-almost surely.
Rearranging the terms in \eqref{eq:veri} and taking the expected value yields
\begin{eqnarray*}
h(z) \geq \E[\e^{-rt}h(\Delta_z^{U}(t))] + \E\bigl[ \int_0^t \e^{-rs}( \beta U(s)-\one_{\{\Delta_z^{U}(s)>d\}}) \id s\bigl]\;.
\end{eqnarray*}
By bounded convergence we find as $t\to\infty$
\begin{equation*}
h(z) \geq \E\bigl[\beta\int_0^\infty \e^{-r t}U(t) \id t - \int_0^\infty \e^{-r t}\onebr{\Delta_z^U(t)>d} \id t \bigr] = v^{U}(z)\;.
\end{equation*}
If $h$ is a solution fulfilling $h'(0)=0$,
we conclude from the steps above for the strategy $U^{*}$ that $h(z) = v^{U^{*}}(z) \le \sup_{u\in?U}v^{U}(z) = v(z)\;.$
\end{proof}

%% file: Propertiesofv.tex
\section{Properties of  the Solution} \label{section:propertiesofv}
\subsection{System of Differential Equations}
Because of the discontinuity of the function $\one_{\{z>d\}}$ in the HJB equation \eqref{HJB}, we have to consider the cases $[0,d]$ and $(d,\infty)$ separately. We introduce the notation $h(z)=f(z)\one_{[0,d]}(z)+g(z)\one_{(d,\infty)}(z)$. Then we can rewrite \eqref{HJB} as system of equations
\begin{eqnarray}
-r f(z) -\mu f'(z) + \frac{\sigma^2}{2} f''(z) + \sup_{u\in[0,u_0]}\{(\beta+f'(z))u\} = 0\;, \qquad &z\in[0,d]&\,, \\
-r g(z) -\mu g'(z) + \frac{\sigma^2}{2} g''(z) + \sup_{u\in[0,u_0]}\{(\beta+g'(z))u\} = 1\;, \qquad &z>d&\,.
\end{eqnarray}
We distinguish the areas $?C_1:=\{z\geq 0:h'(z)< -\beta\}$ and $?C_2:=\{z\geq 0:h'(z)\geq -\beta\}$. For $z\in?C_1$, \eqref{HJB} is of the form
\[
        -rh(z)-\mu h'(z) + \frac{\sigma^2}{2} h''(z) = \one_{\{z>d\}}
\]
and for $z\in?C_2$ of the form
\[
        -rh(z)-(\mu-u_0) h'(z) + \frac{\sigma^2}{2} h''(z) = \one_{\{z>d\}} - \beta u_0\;.
\]
Thus, we have to distinguish four different differential equations. The general solutions are given by
\begin{alpheq}\label{sol}
\begin{eqnarray}
\ubar{f}(z) &=&  A_1\e^{\theta_1(0)z} + B_1\e^{-\theta_2(0)z}\,, \quad z\in?C_1\cap[0,d]\,, \label{sol1a}\\
\bar{f}(z) &=& \frac{\beta u_0}{r} + A_2\e^{\theta_1(u_0)z} + B_2\e^{-\theta_2(u_0)z}\,, \quad z\in?C_2\cap[0,d]\,, \label{sol1b}\\
\ubar{g}(z) &=& -\frac{1}{r} + a_1\e^{\theta_1(0)z} + b_1\e^{-\theta_2(0)z}\,, \label{sol2a}\quad z\in?C_1\cap(d,\infty)\,, \\
\bar{g}(z) &=& \frac{\beta u_0-1}{r} + a_2\e^{\theta_1(u_0)z} + b_2\e^{-\theta_2(u_0)z}\,, \label{sol2b}\quad z\in?C_2\cap(d,\infty)\,,
\end{eqnarray}
\end{alpheq}
where
\begin{eqnarray*}
\theta_1(u) &=& \frac{\sqrt{(\mu-u)^2+2r\sigma^2}+(\mu-u)}{\sigma^2} > 0 \\
\noalign{\text{\rm and}}\\
\theta_2(u) &=& \frac{\sqrt{(\mu-u)^2+2r\sigma^2}-(\mu-u)}{\sigma^2} > 0
\end{eqnarray*}
are such that $\kappa\in\{\theta_1(u),-\theta_2(u)\}$ solves the equation
\[
-r - (\mu-u) \kappa +\frac{\sigma^2}{2} \kappa^2 = 0\;.
\]
In the next section, we will discuss more detailed properties of the solutions.

\subsection{Form of the Solution}

In view of Theorem \ref{verification}, we are looking for a decreasing solution. In particular, the solution will be strictly decreasing, as shown in the following
\begin{lemma} \label{lemma:strictlydecreasing}
If $h:[0,\infty)\to[(\beta u_0 -1)/r,\beta u_0/r]$ is a decreasing solution to \eqref{HJB}, then it is strictly decreasing.
\end{lemma}
\begin{proof}
Assume that there exists an interval $[a,b]\subset[0,\infty)$ such that $h$ is
constant on $[a,b]$. First, we can assume $[a,b]\subset[0,d]$. We then
conclude from \eqref{HJB} that $h(z)=\beta u_0/r$ for all $z\in[a,b]$. Since
$h$ is decreasing and bounded from above by $h(z)\leq\beta u_0/r$, it must be
the case that $h(z)=\beta u_0/r$ for all $z\in[0,b]$. We can assume that we
had chosen the maximal possible $b$. If $b<d$, 
then the solution $h$ on $[b,b+\epsilon]$ is of the form \eqref{sol1b}, that is
\[
h(z) = \frac{\beta u_0}{r} + A_2\e^{\theta_1(u_0)z}+B_2\e^{-\theta_2(u_0)z}\;.
\]
By plugging in the initial conditions $h(b)=\beta u_0/r$ and $h'(b)=0$ we find
$A_2=B_2=0$ and therefore $h(z)=\beta u_0/r$ for all $z\in[0,b+\epsilon]$.
This contradicts the assumption that $b$ is maximal. Thus $h(z) = \beta u_0/r$ for all $z\in[0,d]$. For $\epsilon>0$ small, we extend the solution on $[d,d+\epsilon]$. Since $h'(d)=0$, the solution is of the form \eqref{sol2b}. By $h(d)=\beta u_0/r$ and $h'(d)=0$ we determine the constants $a_2$ and $b_2$ and find
\[
h(z) = \frac{\beta u_0-1}{r} + \frac{\theta_2(u_0)\e^{\theta_1(u_0)(z-d)}+\theta_1(u_0)\e^{-\theta_2(u_0)(z-d)}}{r(\theta_1(u_0)+\theta_2(u_0))}\;, \quad z\in[d,d+\epsilon]\;.
\]
But
\[
h'(z) = \frac{\theta_1(u_0)\theta_2(u_0)}{r(\theta_1(u_0)+\theta_2(u_0))}(\e^{\theta_1(u_0)(z-d)}-\e^{-\theta_2(u_0)(z-d)}) > 0\;,
\]
which contradicts the assumption that $h$ is decreasing. Analogously, we can show that if $h$ is constant on an interval $[a,b]\subset[d,\infty)$, then $h(z)=(\beta u_0 -1)/r$ for all $z\in[d,\infty)$ and the solution $h$ on a small interval $[d-\epsilon,d]$ is increasing, which again leads to a contradiction. This proves the assertion.
\end{proof}
Next, we show that the second derivative $h''(z)$ does not vanish on a
non-empty open interval:
\begin{lemma} \label{lemma:fnotlinear}
Let $h$ be a decreasing solution to \eqref{HJB}. Then there exists no non-empty interval $(a,b)\subset[0,\infty)$ such that $h''(z)=0$ for all $z\in(a,b)$.
\end{lemma}
\begin{proof}
We note that $h$ has to be strictly decreasing by Lemma \ref{lemma:strictlydecreasing}.
Assume that there exists a non-empty interval $(a,b)$ such that $h''(z)=0$ for
all $z\in(a,b)$. 
The HJB equation \eqref{HJB} is then of the form
\begin{equation}
-rh(z)-\mu h'(z) + \sup_{u\in[0,u_0]}\{(\beta+h'(z))u\} = \one_{\{z>d\}}\;. \label{eq:HJB:f''zero}
\end{equation}
Since $h'(z)$ is constant on the interval $(a,b)$, also $h(z)$ has to be
constant contradicting Lemma~\ref{lemma:strictlydecreasing}. This proves the
assertion.
\end{proof}
In order to construct a solution to \eqref{HJB}, we have to combine the different types of solutions. The following lemma is crucial, since it states that we can only glue the solutions together in a specific way.
\begin{lemma} \label{lemma:formofh}
If $h:[0,\infty)\to[(\beta u_0-1)/r,\beta u_0/r]$ is a decreasing solution to \eqref{HJB} with $h'(0)=0$, then it is of the form
\begin{equation*}
h(z) = \bar{f}(z)\one_{[0,z_f\wedge d]}(z) +\ubar{f}(z)\one_{(z_f\wedge d,d]}(z) +
\ubar{g}(z)\one_{(d,z_g\vee d]}(z) +\bar{g}(z)\one_{[z_g\vee d,\infty)}(z)\;,
\end{equation*}
where $z_f \in (0,d]$ and $z_g\in [d,\infty)$.
\end{lemma}
\begin{proof}
First note that $h$ is strictly decreasing by Lemma \ref{lemma:strictlydecreasing}.
We start by showing that $\ubar{g}$ can not be a solution on $(z_g,\infty)$ for
some $z_g\geq d$. Assume that this holds. This is only the case if
$\ubar{g}'(z)\leq -\beta$ for all
$z>z_g$. Since the value function is bounded and decreasing, we have $a_1=0$
and $b_1>0$. But then we find that
\[
-\beta \geq \ubar{g}'(z) = -\theta_2(0) b_1 \e^{-\theta_2(0)z}
\]
for all $z>z_g$, which is impossible for $b_1<\infty$. Thus, either
it exists $z_g\geq d$, such that $h(z)=\bar{g}(z)$ for all $z\geq z_g$ or the
two solutions $\ubar{g}$ and $\bar{g}$ alternate. Next, we show that only the former
can apply.
At a point, where we glue the two possible solutions together, we must have
$\ubar{g}'(z)=\bar{g}'(z) = -\beta$. At such a point we get 
\begin{equation}
\frac12 \sigma^2 h''(z) = r h(z)- \mu\beta +1\;, \label{eq:1}
\end{equation}
implying that for two such points $z_g < \tilde z_g$ we
have $h''(z_g) > h''(\tilde z_g)$. Since we have to switch between $\{h'(z)<-\beta\}$ and $\{h'(z)\geq -\beta\}$ at $z_g$ and $\tilde{z}_g$, this is only possible if $h''(z_g) \ge 0$ and
$h''(\tilde z_g) \le 0$. In view of \eqref{eq:1}, there cannot exist another
$\hat{z}_g>\tilde{z}_g$ where we switch back to $\{h'(z)\geq -\beta\}$. But,
$\ubar{g}$ cannot be the solution on $(\tilde{z}_g,\infty)$, which proves the
assertion for $z \ge d$.

By the boundary condition $h'(0)=0$ we have $h(z) = \bar{f}(z)$ on some interval
$[0,\epsilon)$. An analogous argument shows then the assertion for $z \in
[0,d]$. 
\end{proof}
By using the boundedness of $h(z)$ and the boundary condition, we can specify $\bar{f}$ and $\bar{g}$ more precisely. Moreover, we can determine $\lim_{z\to\infty} h(z)$ as stated in Lemma \ref{lemma:propertiesofv}:
\begin{lemma}
The solution to \eqref{HJB} is of the form
\begin{equation*}
\bar{f}(z) = \frac{\beta u_0}{r} +
A_2\bigl(\e^{\theta_1(u_0)z}+\frac{\theta_1(u_0)}{\theta_2(u_0)}
\e^{-\theta_2(u_0)z}\bigr) \quad \text{\rm for}\; z\in[0,z_f]
\end{equation*}
and
\begin{equation*}
\bar{g}(z) = \frac{\beta u_0-1}{r} + b_2 \e^{-\theta_2(u_0)z} \quad \text{for}\; z>z_g\;.
\end{equation*}
In particular, $\lim_{z\to\infty} h(z) = (\beta u_0-1)/r$.
\hfill\qed
\end{lemma}

%% file: ADirectApproach.tex
\section{A Direct Approach} \label{section:adirectapproach}

By Lemma \ref{lemma:formofh} we know how a solution to \eqref{HJB} looks
like. In principle, we just have to solve Equation \eqref{HJB} with the
boundary conditions at $0$, $d$ and $\infty$. This works without problems in
the case where dividends are always paid, see
Theorem~\ref{lemma:solution:beta>zeta:1}. If $z_f$, $z_g$ are different from
$d$ a solution is hard. We will show in Section \ref{section:constructionofsolution} that a solution exists
and give a procedure how to determine $v(d)$.

For now, we start with the case $U(t) = u_0$.
\begin{theorem} \label{lemma:solution:beta>zeta:1}
Let
\begin{alpheq}\label{eq:casexi1}
\begin{eqnarray}
\bar{f}(z) &=& \frac{\beta u_0}{r} -
\frac{1}{r}\frac{\theta_2(u_0)\e^{\theta_1(u_0)z}+\theta_1(u_0)\e^{-\theta_2(u_0)z}}{(\theta_1(u_0)+\theta_2(u_0))\e^{\theta_1(u_0)d}}\;,
\quad z\in[0,d]\;, \quad \text{\rm and} \label{eq:casexi1a}\\
\bar{g}(z) &=& \frac{\beta u_0 -1}{r} + \frac{1}{r} \frac{\theta_1(u_0)(\e^{\theta_1(u_0)d}-\e^{-\theta_2(u_0)d})}{(\theta_1(u_0)+\theta_2(u_0))\e^{\theta_1(u_0)d}}\e^{-\theta_2(u_0)(z-d)}\;, \quad z>d\;. \label{eq:casexi1b}
\end{eqnarray}
\end{alpheq}
Then, the value function is given by
\begin{equation*}
v(z) = \bar{f}(z)\one_{[0,d]}(z) + \bar{g}(z)\one_{(d,\infty)}(z)\;,
\end{equation*}
if and only if
\begin{equation}
\beta \ge \zeta := \frac{\theta_1(u_0)\theta_2(u_0)(\e^{\theta_1(u_0)d}-\e^{-\theta_2(u_0)d})}{r(\theta_1(u_0)+\theta_2(u_0))\e^{\theta_1(u_0)d}}>0\;. \label{zeta}
\end{equation}
In this case, an optimal strategy is $U(t) = u_0$ for all $t\geq 0$.
\end{theorem}
\begin{proof}
We can calculate the value of the strategy $U(t) = u_0$. By the verification
theorem it then suffices to show that the corresponding value function solves
\eqref{HJB} if and only if the condition is fulfilled. The solutions
\eqref{sol1b} and \eqref{sol2b} with the boundary condition $\bar{f}'(0) = 0$,
$\bar{g}(\infty) < \infty$, $\bar{f}(d) = \bar{g}(d)$ and $\bar{f}'(d) =
\bar{g}'(d)$ yield the solutions \eqref{eq:casexi1}.
In order that $h(z) = \bar{f}(z)\one_{[0,d]}(z) +
\bar{g}(z)\one_{(d,\infty)}(z)$ is a solution to the HJB equation \eqref{HJB},
we need $h'(z) \geq -\beta$ for all $z\geq 0$. Note that $\bar{f}(z)$ is
concave and $\bar g(z)$ is convex.
We consider $\bar{f}$ and $\bar{g}$ separately. For $z\in[0,d]$ we find
\begin{eqnarray*}
\bar{f}'(z) &=& -\frac{1}{r}\frac{\theta_1(u_0)\theta_2(u_0)(\e^{\theta_1(u_0)z}-\e^{-\theta_2(u_0)z})}{(\theta_1(u_0)+\theta_2(u_0))\e^{\theta_1(u_0)d}} \\
&\geq& -\frac{1}{r}\frac{\theta_1(u_0)\theta_2(u_0)(\e^{\theta_1(u_0)d}-\e^{-\theta_2(u_0)d})}{(\theta_1(u_0)+\theta_2(u_0))\e^{\theta_1(u_0)d}} = -\zeta \geq -\beta \;.
\end{eqnarray*}
For $z>d$ we find
\begin{eqnarray*}
\bar{g}'(z) &=& -\frac{1}{r} \frac{\theta_1(u_0)\theta_2(u_0)(\e^{\theta_1(u_0)d}-\e^{-\theta_2(u_0)d})}{(\theta_1(u_0)+\theta_2(u_0))\e^{\theta_1(u_0)d}} \e^{-\theta_2(u_0)(z-d)} \\
&\geq& -\frac{1}{r} \frac{\theta_1(u_0)\theta_2(u_0)(\e^{\theta_1(u_0)d}-\e^{-\theta_2(u_0)d})}{(\theta_1(u_0)+\theta_2(u_0))\e^{\theta_1(u_0)d}} = -\zeta \geq -\beta \;.
\end{eqnarray*}
This shows that $h(z)$ solves the HJB equation \eqref{HJB} with the optimiser
$u^{*}(z) = u_0$ for all $z\geq 0$.
\end{proof}
\begin{rem} \label{rem:zeta}
Note that the function $u_0\mapsto\zeta(u_0)$ attains its maximum at $u_0=\mu$
and that $\zeta(\mu-u_0)=\zeta(\mu+u_0)$ for all $u_0\geq 0$. On the other
hand, the function $\sigma^2\mapsto\zeta(\sigma^2)$ is decreasing. We will
give an interpretation of $\zeta$ and its behaviour in Section
\ref{section:numerical}. Obviously, if $u_0$ is small, the loss of drift
is small and therefore one can risk paying dividends for a smaller $\beta$. If
$u_0$ is large, the profit of the dividend payments is larger than the loss by
the drawdown already for a small $\beta$.
\begin{figure} \label{fig:zeta}
\centering
\includegraphics[width=0.5\textwidth]{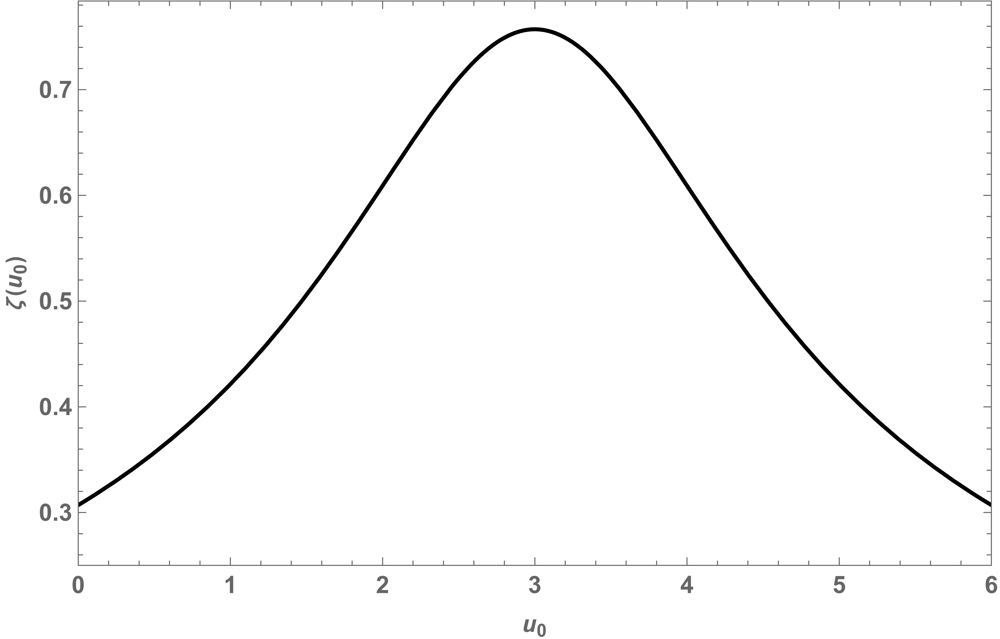}
\caption{Plot of $u_0\mapsto\zeta(u_0)$ for parameters $\mu=3$, $\sigma=2$, $d=5$ and $r=0.2$.}
\end{figure}
\hfill$\blacksquare$
\end{rem}
One could try the same direct approach to construct a solution of the form
\begin{equation}
        h(z) = \bar{f}(z)\one_{[0,z_f]}(z)+\ubar{f}(z)\one_{(z_f,d]}(z)+\ubar{g}(z)\one_{(d,z_g]}(z)+\bar{g}(z)\one_{(z_g,\infty)}(z)\;, \label{eq:soltype2}
\end{equation}
where $z_f\in(0,d)$ and $z_g\in[d,\infty)$ or $z_f\in(0,d]$ and
$z_g\in(d,\infty)$. It then turns out that it is hard to give an explicit
expression for $z_f$ and $z_g$. Because we did not yet show that a solution to
\eqref{HJB} exists, we take a different approach, which is presented in the next section.
\begin{figure} \label{fig:divdom}
\centering
\includegraphics[width=0.7\textwidth]{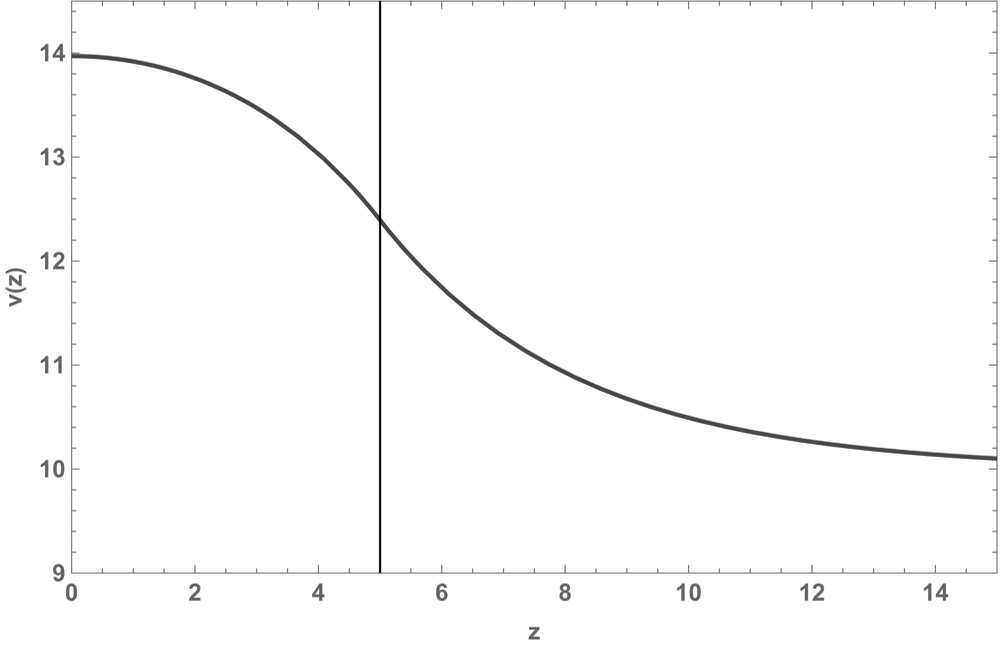}
\caption{Plot of $v(z)$ for parameters $\mu=3$, $\sigma=2$, $u_0=3$, $d=5$, $r=0.2$ and $\beta=1>\zeta=0.757105$.}
\end{figure}

%% file: ConstructionofSolution.tex
\section{Construction of a Solution} \label{section:constructionofsolution}
\subsection{Breakdown into two Sub-Problems}
Let us now consider a slightly different problem: For $C\in((\beta u_0-1)/r,\beta u_0/r)$ we want to find bounded and twice continuously differentiable functions $f_C$ on $[0,d]$ and $g_C$ on $[d,\infty)$, such that
 \begin{eqnarray}
 -rf_C(z)-\mu f_C'(z) + \frac{\sigma^2}{2} f_C''(z) + \sup_{u\in[0,u_0]}\{(\beta+f_C'(z))u\} = 0 \label{problem1}
 \end{eqnarray}
with the boundary conditions $f_C'(0)=0$ and $f_C(d)=C$ as well as
 \begin{eqnarray}
 -rg_C(z)-\mu g_C'(z) + \frac{\sigma^2}{2} g_C''(z) + \sup_{u\in[0,u_0]}\{(\beta+g_C'(z))u\} = 1 \label{problem2}
\end{eqnarray}
with the boundary condition $g_C(d)=C$. The possible solutions are again of the
types \eqref{sol}. We now want to show that each of the two subproblems has a
solution. In order to find a solution to our original problem on
$[0,\infty)$, we will show in the end that it is possible to choose $C$ in
such a way that $f_C'(d)=g_C'(d)$. But for the moment, we will leave this condition aside. The following two constants will be important:
\begin{eqnarray}
\xi_1(\beta) &:=& \frac{\beta u_0}{r} - \beta \frac{\e^{\theta_1(u_0)d}+\frac{\theta_1(u_0)}{\theta_2(u_0)}\e^{-\theta_2(u_0)d}}{\theta_1(u_0)(\e^{\theta_1(u_0)d}-\e^{-\theta_2(u_0)d})} < \frac{\beta u_0}{r} \\
 \xi_2(\beta) &:=& \frac{\beta u_0 - 1}{r} + \frac{\beta}{\theta_2(u_0)} > \frac{\beta u_0 - 1}{r}\;.
\end{eqnarray}
Direct calculation yields the following
\begin{lemma} \label{lemma:connectionzetaxi}
We have 
\[\xi_1(\beta) \lesseqqgtr \xi_2(\beta) \quad \Leftrightarrow \quad \zeta
\lesseqqgtr \beta\;, \]
where $\zeta$ is defined in \eqref{zeta}.\hfill\qed
\end{lemma}

\subsection{The Dividend-Dominated Case}

In the following, we will reproduce the result of Theorem
\ref{lemma:solution:beta>zeta:1} using the new approach. That is, $U(t) =
u_0$ is optimal if $\beta\geq\zeta$ and $C \in [\xi_1(\beta),\xi_2(\beta)]$.
\begin{prop} \label{prop:Cbiggerxi}
If $\xi_1(\beta) \le C< \beta u_0/r$, then the function
\begin{alpheq}
\begin{equation}
f_C(z) = \frac{\beta u_0}{r} - \frac{\frac{\beta u_0}{r}-C}{\e^{\theta_1(u_0)d}+\frac{\theta_1(u_0)}{\theta_2(u_0)}\e^{-\theta_2(u_0)d}}\bigl( \e^{\theta_1(u_0)z}+\frac{\theta_1(u_0)}{\theta_2(u_0)}\e^{-\theta_2(u_0)z}\bigr)\;, \qquad z\in[0,d] \label{eq:f:beta>zeta}
\end{equation}
solves \eqref{problem1} with $f_C'(0)=0$ and $f_C(d)=C$. 
If $(\beta u_0 - 1)/r < C\leq\xi_2(\beta)$, then the function
\begin{equation}
g_C(z) = \frac{\beta u_0 - 1}{r}+\bigl(C - \frac{\beta u_0 -1}{r} \bigr)\e^{-\theta_2(u_0)(z-d)}\;, \qquad z\in[d,\infty) \label{eq:g:beta>zeta}
\end{equation}
\end{alpheq}
solves \eqref{problem2} with $g_C(d)=C$ and $g_C(\infty)<\infty$.
 \end{prop}
 \begin{proof}
It is easy to see that $f_C(d)=g_C(d)=C$, $f_C'(0)=0$, $g_C(\infty)<\infty$, $f_C$ is
concave, and $g_C$ is convex. $f_C$ is of type \eqref{sol1b} and $g_C$ is of type \eqref{sol2b}. Therefore, it remains to show that $f_C'(z)\geq - \beta$ for all $z\in[0,d]$ and $g_C'(z)\geq - \beta$ for all $z\in[d,\infty)$. Note that for $z\in[0,d]$
\begin{eqnarray*}
f_C'(z) &=& \frac{C-\frac{\beta u_0}{r}}{\e^{\theta_1(u_0)d}+\frac{\theta_1(u_0)}{\theta_2(u_0)}\e^{-\theta_2(u_0)d}}\theta_1(u_0)(\e^{\theta_1(u_0)z}-\e^{-\theta_2(u_0)z}) \\
&\geq&  \frac{C-\frac{\beta u_0}{r}}{\e^{\theta_1(u_0)d}+\frac{\theta_1(u_0)}{\theta_2(u_0)}\e^{-\theta_2(u_0)d}}\theta_1(u_0)(\e^{\theta_1(u_0)d}-\e^{-\theta_2(u_0)d}) \\
&\geq&  \frac{\xi_1(\beta)-\frac{\beta u_0}{r}}{\e^{\theta_1(u_0)d}+\frac{\theta_1(u_0)}{\theta_2(u_0)}\e^{-\theta_2(u_0)d}}\theta_1(u_0)(\e^{\theta_1(u_0)d}-\e^{-\theta_2(u_0)d}) = -\beta
\end{eqnarray*}
and for $z\in[d,\infty)$
\begin{eqnarray*}
g_C'(z) &=& -\theta_2(u_0)\bigl(C-\frac{\beta u_0 -1}{r} \bigr)\e^{-\theta_2(u_0)(z-d)} \geq -\theta_2(u_0)\bigl(C-\frac{\beta u_0 -1}{r} \bigr) \\
&\geq& -\theta_2(u_0)\bigl(\xi_2(\beta)-\frac{\beta u_0 -1}{r} \bigr) = -\beta
\end{eqnarray*}
and the assertion follows.
\end{proof}
\begin{rem}
Note that $C = v(d)$ in Theorem~\ref{lemma:solution:beta>zeta:1} fulfils $C
\in [\xi_1(\beta),\xi_2(\beta)]$ in the case $\beta \ge \zeta$.
\hfill$\blacksquare$
\end{rem}
We can interpret our result in the following way: If the weight of the
dividends $\beta$ is relatively large we do not care of the
drawdown size anymore and pay dividends at the maximal rate. We will see in
the next section that $\zeta$ is a 
critical value in the sense that the optimal strategy is different for
$\beta<\zeta$. We will show that in this case, the solution is of the form
\eqref{eq:soltype2} with $z_f\in(0,d)$ and $z_g\in(d,\infty)$.
\begin{rem} \label{rem:1}
We will see below that $z_f$ and $z_g$ converge to $d$ as $\beta\uparrow
\zeta$. We therefore have a continuous transition from $\beta < \zeta$ to
$\beta > \zeta$ also for the optimal strategies.
\hfill $\blacksquare$
\end{rem}

\subsection{The Drawdown-Sensitive Case}

The following two propositions complement Proposition \ref{prop:Cbiggerxi}. Note that Lemma \ref{lemma:formofh} still holds in the setting of this section, since the proof did not require the condition $\ubar{f}'(d)=\ubar{g}'(d)$.
 \begin{prop} \label{prop:zf}
If $C<\xi_1(\beta)$, then there exists a unique $z_f^C\in(0,d)$ such that the function $f_C(z) = \bar{f}_C(z)\one_{[0,z_f^C]}(z)+\ubar{f}_C\one_{(z_f^C,d]}(z)$ solves \eqref{problem1} with $f_C'(0)=0$ and $f_C(d)=C$, where
\begin{alpheq}
\begin{eqnarray}
\bar{f}_C(z) &=& \frac{\beta u_0}{r} - \beta
\frac{\e^{\theta_1(u_0)z}+\frac{\theta_1(u_0)}{\theta_2(u_0)}\e^{-\theta_2(u_0)z}}{\theta_1(u_0)(\e^{\theta_1(u_0)z_f^C}-\e^{-\theta_2(u_0)z_f^C})}\qquad
\label{flower} \\
\noalign{\text{\rm and}\nonumber}\\
\ubar{f}_C(z) &=& \frac{C\theta_2(0)\e^{-\theta_2(0)(z_f^C-d)}-\beta}{\theta_1(0)\e^{\theta_1(0)z_f^C}+\theta_2(0)\e^{(\theta_1(0)+\theta_2(0))d}\e^{-\theta_2(0)z_f^C}}(\e^{\theta_1(0)z}-\e^{(\theta_1(0)+\theta_2(0))d}\e^{-\theta_2(0)z}) \nonumber\\
&&\hskip1cm +\; C\e^{-\theta_2(0)(z-d)}\;. \label{fupper}
\end{eqnarray}
\end{alpheq}
Moreover, the function $C\mapsto z_f^C$ is increasing and $z_f^{\xi_1(\beta)}=d$. \end{prop}
\begin{proof}
From Proposition~\ref{prop:Cbiggerxi} we know that a function of type
\eqref{sol1b} cannot be a solution to \eqref{problem1}, since the condition
$C<\xi_1(\beta)$ violates $f_C'(z)\geq -\beta$ for all $z\in[0,d]$. By Lemma
\ref{lemma:formofh} the solution is then of the form
\[
f_C(z) = \bar{f}_C(z)\one_{[0,z_f^C]}(z)+\ubar{f}_C\one_{(z_f^C,d]}(z)\;,
\]
where $\bar{f}_C$ is of type \eqref{sol1b} and $\ubar{f}_C$ is of type \eqref{sol1a}. By using the conditions $\bar{f}_C'(0)=0$, $\ubar{f}(d)=C$ and $\bar{f}_C'(z_f^C)=\ubar{f}_C'(z_f^C)=-\beta$, we find the expressions \eqref{flower} and \eqref{fupper}. Since $z_f^C\mapsto\bar{f}_C(z_f^C)-\ubar{f}_C(z_f^C)$ is continuous and 
\[
(\bar{f}_C(z_f^C)-\ubar{f}_C(z_f^C))\big\vert_{z_f^C=0} = -\infty <0
\]
and
\[
(\bar{f}_C(z_f^C)-\ubar{f}_C(z_f^C))\big\vert_{z_f^C=d} = \xi_1(\beta) - C > 0\;,
\]
there exists $z_f^C\in(0,d)$ such that
$\bar{f}_C(z_f^C)=\ubar{f}_C(z_f^C)$. Note that this and the HJB equation
imply that $\bar{f}_C''(z_f^C)=\ubar{f}_C''(z_f^C)$. The latter can be
rewritten into the equation
\begin{eqnarray*}
-\beta\frac{\theta_1(u_0)\e^{\theta_1(u_0)z_f^C}+\theta_2(u_0)\e^{-\theta_2(u_0)z_f^C}}{\e^{\theta_1(u_0)z_f^C}-\e^{-\theta_2(u_0)z_f^C}} &=& C\;\frac{\theta_1(0)\theta_2(0)(\theta_1(0)+\theta_2(0))\e^{\theta_1(0)z_f^C}\e^{-\theta_2(0)(z_f^C-d)}}{\theta_1(0)\e^{\theta_1(0)z_f^C}+\theta_2(0)\e^{(\theta_1(0)+\theta_2(0))d}\e^{-\theta_2(0)z_f^C}} \\
&\hspace{1cm}& - \beta \frac{\theta_1(0)^2\e^{\theta_1(0)z_f^C}-\theta_2(0)^2\e^{(\theta_1(0)+\theta_2(0))d}\e^{-\theta_2(0)z_f^C}}{\theta_1(0)\e^{\theta_1(0)z_f^C}+\theta_2(0)\e^{(\theta_1(0)+\theta_2(0))d}\e^{-\theta_2(0)z_f^C}}
\end{eqnarray*}
or equivalently
\begin{equation}
C = \frac{\beta}{\theta_1(0)+\theta_2(0)} \Bigl( \frac{\rho_1(z_f^C)}{\theta_2(0)\e^{\theta_2(0)d}}+\frac{\rho_2(z_f^C)}{\theta_1(0)\e^{-\theta_1(0)d}}\Bigr)\;, \label{eq:zfC}
\end{equation}
where
\begin{eqnarray*}
\rho_1(z) &:=&
\frac{(\theta_1(0)-\theta_1(u_0))\e^{(\theta_1(u_0)+\theta_2(0))z}-(\theta_1(0)+\theta_2(u_0))\e^{(\theta_2(0)-\theta_2(u_0))z}}{\e^{\theta_1(u_0)z}-\e^{-\theta_2(u_0)z}}\\ 
\noalign{\noindent\text{and}} \\
\rho_2(z) &:=& \frac{(\theta_2(0)-\theta_2(u_0))\e^{-(\theta_1(0)+\theta_2(u_0))z}-(\theta_2(0)+\theta_1(u_0))\e^{(\theta_1(u_0)-\theta_1(0))z}}{\e^{\theta_1(u_0)z}-\e^{-\theta_2(u_0)z}}\;.
\end{eqnarray*}
In order to show uniqueness of $z_f^C$, we show that $z\mapsto\rho_1(z)$ and $z\mapsto\rho_2(z)$ are increasing. Differentiating with respect to $z$ and rearranging the terms yields for $\rho_1$:
\begin{eqnarray*}
\rho_1'(z)&=&\frac{\e^{\theta_2(0)z}}{(\e^{\theta_1(u_0)z}-\e^{-\theta_2(u_0)z})^2}\bigl[\e^{(\theta_1(u_0)-\theta_2(u_0))z}\bigl(\theta_1(u_0)^2+\theta_2(u_0)(\theta_2(u_0)-\theta_2(0))+\theta_1(u_0)\theta_2(0) \bigr) \\
&& \hskip2cm {}+ \e^{2\theta_1(u_0)z}\theta_2(0)(\theta_1(0)-\theta_1(u_0)) + \e^{-2\theta_2(u_0)z}\theta_2(0)(\theta_1(0)+\theta_2(u_0))\bigr] > 0\;,
\end{eqnarray*}
where we used $\theta_1(0)>\theta_1(u_0)$ and $\theta_2(0)<\theta_2(u_0)$. For $\rho_2$ we find
\begin{eqnarray*}
\rho_2'(z)&=&\frac{\e^{-\theta_1(0)z}}{(\e^{\theta_1(u_0)z}-\e^{-\theta_2(u_0)z})^2}\bigl[\e^{-2\theta_2(u_0)z}\theta_1(0)\theta_2(0)+\bigl(\e^{(\theta_1(u_0)-\theta_2(u_0))z}-\e^{-2\theta_2(u_0)z}\bigr)\theta_1(0)\theta_2(u_0) \\
&&\hskip.5cm {}+\bigl(\e^{2\theta_1(u_0)z}-\e^{(\theta_1(u_0)-\theta_2(u_0))z}\bigr)\theta_1(0)(\theta_1(u_0)+\theta_2(0))\\
&&\hskip1cm {}+ \e^{(\theta_1(u_0)-\theta_2(u_0))z}(\theta_1(u_0)^2+\theta_2(u_0)^2+\theta_1(0)\theta_2(0)) \bigr]>0\;.
\end{eqnarray*}
From this we can conclude that the right side of equation \eqref{eq:zfC} is
strictly increasing in $z_f^C$. This shows that $z_f^C$ has to be unique and
that $C\mapsto z_f^C$ is increasing. Moreover, setting $z_f^C=d$ in
\eqref{eq:zfC} yields for the right-hand side
\begin{eqnarray*}
C &=& \frac{\beta}{\theta_1(0)\theta_2(0)}\Bigl(\theta_1(0)-\theta_2(0)-\frac{\theta_1(u_0)\e^{\theta_1(u_0)d}+\theta_2(u_0)\e^{-\theta_2(u_0)d}}{\e^{\theta_1(u_0)d}-\e^{-\theta_2(u_0)d}} \Bigr) \\
&=& \beta\Bigl(\frac{\mu}{r}-\frac{\e^{\theta_1(u_0)d}+\frac{\theta_1(u_0)}{\theta_2(u_0)}\e^{-\theta_2(u_0)d}}{\theta_1(u_0)(\e^{\theta_1(u_0)d}-\e^{-\theta_2(u_0)d})} + \frac{\theta_1(u_0)-\theta_2(u_0)}{\theta_1(u_0)\theta_2(u_0)} \Bigr) \\
&=& \frac{\beta u_0}{r}-\beta
\frac{\e^{\theta_1(u_0)d}+\frac{\theta_1(u_0)}{\theta_2(u_0)}\e^{-\theta_2(u_0)d}}{\theta_1(u_0)(\e^{\theta_1(u_0)d}-\e^{-\theta_2(u_0)d})}
= \xi_1(\beta)
\end{eqnarray*}
and therefore $z_f^{\xi_1(\beta)}=d$.
Finally, we need to show that $f_C(z) = \bar{f}_C(z)\one_{[0,z_f^C]}(z)+\ubar{f}_C\one_{(z_f^C,d]}(z)$ is a solution to \eqref{problem1}, that is, $\bar{f}_C'(z)\geq -\beta$ for all $z\in[0,z_f^C]$ and $\ubar{f}_C'(z) \leq -\beta$ for all $z\in[z_f^C,d]$. We note that
\[
\bar{f}_C'(z) = -\beta \frac{\e^{\theta_1(u_0)z}-\e^{-\theta_2(u_0)z}}{\e^{\theta_1(u_0)z_f^C}-\e^{-\theta_2(u_0)z_f^C}} \geq - \beta\;, \qquad z\in[0,z_f^C]\;.
\]
Moreover,       
\begin{eqnarray*}
\ubar{f}_C'(z) &=& \frac{C\theta_2(0)\e^{-\theta_2(0)(z_f^C-d)}-\beta}{\theta_1(0)\e^{\theta_1(0)z_f^C}+\theta_2(0)\e^{(\theta_1(0)+\theta_2(0))d}\e^{-\theta_2(0)z_f^C}}(\theta_1(0)\e^{\theta_1(0)z}+\theta_2(0)\e^{(\theta_1(0)+\theta_2(0))d}\e^{-\theta_2(0)z}) \\
&&\hskip2cm {}-\theta_2(0) C\e^{-\theta_2(0)(z-d)}\;.
\end{eqnarray*}
By replacing $C$ with the expression from \eqref{eq:zfC}, we can rewrite $\ubar{f}_C'(z)$ as follows:
\begin{equation}
\ubar{f}_C'(z) = -\beta \frac{\tau(z,z_f^C)}{(\theta_1(0)+\theta_2(0))\e^{\theta_1(0)z_f^C}(\e^{\theta_1(u_0)z_f^C}-\e^{-\theta_2(u_0)z_f^C})}\;, \label{eq:dfc}
\end{equation}
 where 
 \begin{eqnarray}
\tau(z,z_f^C) &:=& \e^{\theta_1(0)z}\bigl((\theta_1(u_0)+\theta_2(0))\e^{\theta_1(u_0)z_f^C}-(\theta_2(0)-\theta_2(u_0))\e^{-\theta_2(u_0)z_f^C} \bigr) \nonumber \\
&+& \e^{-\theta_2(0)z}\e^{(\theta_1(0)+\theta_2(0))z_f^C}\bigl((\theta_1(0)-\theta_1(u_0))\e^{\theta_1(u_0)z_f^C}-(\theta_1(0)+\theta_2(u_0))\e^{-\theta_2(u_0)z_f^C} \bigr)\;. \label{eq:tau}
 \end{eqnarray}
Since 
\begin{eqnarray*}
\frac{\id{}}{\id z}\tau(z,z_f^C) &=&  \theta_1(0)\e^{\theta_1(0)z}\bigl((\theta_1(u_0)+\theta_2(0))\e^{\theta_1(u_0)z_f^C}-(\theta_2(0)-\theta_2(u_0))\e^{-\theta_2(u_0)z_f^C} \bigr) \\
&& \hspace{0.25cm} {}- \theta_2(0)\e^{-\theta_2(0)z}\e^{(\theta_1(0)+\theta_2(0))z_f^C}\bigl((\theta_1(0)-\theta_1(u_0))\e^{\theta_1(u_0)z_f^C}-(\theta_1(0)+\theta_2(u_0))\e^{-\theta_2(u_0)z_f^C} \bigr) \\
&=& \theta_1(u_0)\e^{\theta_1(u_0)z_f^C}(\theta_1(0)\e^{\theta_1(0)z}+\theta_2(0)\e^{(\theta_1(0)+\theta_2(0))z_f^C}\e^{-\theta_2(0)z})\\
&& \hspace{0.25cm} {}+\theta_1(0)(\theta_2(u_0)-\theta_2(0))\e^{\theta_1(0)z}\e^{-\theta_2(u_0)z_f^C}+\theta_1(0)\theta_2(0)\e^{\theta_1(u_0)z_f^C}(\e^{\theta_1(0)z}-\e^{\theta_1(0)z_f^C}\e^{-\theta_2(0)(z-z_f^C)})\\
&& \hspace{0.25cm} {}+\theta_2(0)(\theta_1(0)+\theta_2(u_0))\e^{(\theta_1(0)+\theta_2(0))z_f^C}\e^{-\theta_2(0)z}\e^{-\theta_2(u_0)z_f^C} > 0\;,
\end{eqnarray*}
where we used $\theta_2(u_0)>\theta_2(0)$ and $z> z_f^C$, we find that $\tau(z)$ is increasing in $z$ and therefore $\ubar{f}_C'(z) < - \beta$ for all $z\in(z_f^C,d]$.
 \end{proof}
 \begin{prop} \label{prop:zg}
If $C>\xi_2(\beta)$, then there exists a unique $z_g^C\in(d,\infty)$ such that the function $g_C(z)=\ubar{g}_C(z)\one_{[d,z_g^C]}(z)+\bar{g}_C(z)\one_{(z_g^C,\infty)}(z)$ is bounded and solves \eqref{problem2} with $g_C(d)=C$, where
\begin{alpheq}\label{glu}
\begin{eqnarray}
\ubar{g}_C(z) &=& -\frac{1}{r}+\frac{C+\frac{1}{r}-\frac{\beta}{\theta_2(0)}\e^{-\theta_2(0)(d-z_g^C)}}{\e^{\theta_1(0)d}+\frac{\theta_1(0)}{\theta_2(0)}\e^{(\theta_1(0)+\theta_2(0))z_g^C}\e^{-\theta_2(0)d}}\bigr(\e^{\theta_1(0)z}+\frac{\theta_1(0)}{\theta_2(0)}\e^{(\theta_1(0)+\theta_2(0))z_g^C}\e^{-\theta_2(0)z}\bigl)\nonumber\\
&&\hskip1cm +\frac{\beta}{\theta_2(0)}\e^{-\theta_2(0)(z-z_g^C)} \label{glower} \\
\noalign{\noindent\text{\rm and}\nonumber}\\
\bar{g}_C(z) &=& \frac{\beta u_0 - 1}{r}+ \frac{\beta}{\theta_2(u_0)}\e^{-\theta_2(u_0)(z-z_g^C)}\;. \label{gupper}
\end{eqnarray}
\end{alpheq}
Moreover, the function $C\mapsto z_g^C$ is increasing and $z_g^{\xi_2(\beta)}=d$. 
\end{prop}
\begin{proof}
From Proposition \ref{prop:Cbiggerxi} we know that a function of type \eqref{sol2b} cannot be a solution to \eqref{problem2}, since the condition $C>\xi_2(\beta)$ violates $f_C'(z)\geq -\beta$ for all $z\in[0,d]$. By Lemma \ref{lemma:formofh}, the solution is then of the form
\[
g_C(z) = \ubar{g}_C(z)\one_{[d,z_g^C]}(z)+\bar{g}_C\one_{(z_g^C,\infty)}(z)\;,
\]
where $\ubar{g}_C$ is of type \eqref{sol2a} and $\bar{g}_C$ is of type
\eqref{sol2b}. By using $g_C(d)=C$,
$\ubar{g}_C'(z_g^C)=\bar{g}_C'(z_g^C)=-\beta$ and $g_C(\infty)<\infty$, we
find the expressions \eqref{glu}. Since
$z_g^C\mapsto\ubar{g}_C(z_g^C)-\bar{g}_C(z_g^C)$ is continuous and
\[
(\ubar{g}_C(z_g^C)-\bar{g}_C(z_g^C))\big\vert_{z_g^C=d} = C - \xi_2(\beta) > 0\;,
\]
and
\[
(\ubar{g}_C(z_g^C)-\bar{g}_C(z_g^C))\big\vert_{z_g^C=\infty} = -\frac{1}{r} -\frac{\beta}{\theta_1(0)} -\xi_2(\beta) <0\;,
\]
there exists $z_g^C\in(d,\infty)$ such that $\ubar{g}_C(z_g^C)=\bar{g}_C(z_g^C)$. In order to show uniqueness of $z_g^C$, note that the HJB equation \eqref{HJB} and the properties of $z_g^C$ imply $\ubar{g}''_C(z_g^C)=\bar{g}''_C(z_g^C)$. The latter can be rewritten into the equation
\begin{eqnarray*}
\beta\theta_2(u_0) = \frac{C+\frac{1}{r}-\frac{\beta}{\theta_2(0)}\e^{-\theta_2(0)(d-z_g^C)}}{\e^{\theta_1(0)d}+\frac{\theta_1(0)}{\theta_2(0)}\e^{(\theta_1(0)+\theta_2(0))z_g^C}\e^{-\theta_2(0)d}}\theta_1(0)(\theta_1(0)+\theta_2(0))\e^{\theta_1(0)z_g^C} + \beta\theta_2(0)
\end{eqnarray*}
or equivalently
\begin{equation}
C = \frac{\beta(\theta_2(u_0)-\theta_2(0))}{\theta_1(0)(\theta_1(0)+\theta_2(0))}\bigl( \e^{\theta_1(0)(d-z_g^C)}+\frac{\theta_1(0)}{\theta_2(0)}\e^{-\theta_2(0)(d-z_g^C)}\bigr)-\frac{1}{r}+\frac{\beta}{\theta_2(0)}\e^{-\theta_2(0)(d-z_g^C)}\;. \label{eq:zgC}
\end{equation}
The right-hand side is strictly increasing in $z_g^C$. This implies that $z_g^C$ is unique and that $C\mapsto z_g^C$ is increasing. Moreover, for $z_g^C=d$ the right-hand side of \eqref{eq:zgC} becomes 
\begin{eqnarray*}
\lefteqn{\frac{\beta(\theta_2(u_0)-\theta_2(0))}{\theta_1(0)\theta_2(0)}-\frac{1}{r}+\frac{\beta}{\theta_2(0)} = \beta \Bigl[\frac{\theta_2(u_0)+\theta_1(0)-\theta_2(0)}{\theta_1(0)\theta_2(0)} \Bigr] -\frac{1}{r}}\hskip1cm \\
&=& \beta\Bigl[\frac{\sqrt{(\mu-u_0)^2+2r\sigma^2}+\mu+u_0}{2r} \Bigr] -\frac{1}{r}
= \beta\Bigl[\frac{u_0}{r}+\frac{\sigma^2}{2r} \theta_1(u_0)\Bigr]
-\frac{1}{r}\\ &=& \frac{\beta u_0-1}{r}+\frac{\beta}{\theta_2(u_0)} = \xi_2(\beta)\;.
\end{eqnarray*}
This proves that $z_g^{\xi_2(\beta)}=d$. Finally, we need to show that $g_C(z) = \ubar{g}_C(z)\one_{[d,z_g^C]}(z)+\bar{g}_C\one_{(z_g^C,\infty)}(z)$ is a solution to \eqref{problem2}, that is, $\ubar{g}_C'(z)\leq -\beta$ for all $z\in[d,z_g^C]$ and $\bar{g}_C'(z) \geq -\beta$ for all $z\in[z_g^C,\infty)$. For $z\in[z_g^C,\infty)$ we note that
\[
\bar{g}_C'(z) = -\beta\e^{-\theta_2(u_0)(z-z_g^C)} \geq -\beta \;.
\]
For $z\in[d,z_g^C]$ we get
\begin{eqnarray*}
\ubar{g}_C'(z) &=& \frac{C+\frac{1}{r}-\frac{\beta}{\theta_2(0)}\e^{-\theta_2(0)(d-z_g^C)}}{\e^{\theta_1(0)d}+\frac{\theta_1(0)}{\theta_2(0)}\e^{(\theta_1(0)+\theta_2(0))z_g^C}\e^{-\theta_2(0)d}}\theta_1(0)\bigr(\e^{\theta_1(0)z}-\e^{(\theta_1(0)+\theta_2(0))z_g^C}\e^{-\theta_2(0)z}\bigl) -\beta\e^{-\theta_2(0)(z-z_g^C)}
\end{eqnarray*}
By replacing $C$ with the expression in \eqref{eq:zgC}, we find
\begin{eqnarray}
\ubar{g}_C'(z) &=& \beta \frac{\theta_2(u_0)-\theta_2(0)}{\theta_1(0)+\theta_2(0)}(\e^{-\theta_1(0)(z_g^C-z)}-\e^{\theta_2(0)(z_g^C-z)})-\beta\e^{\theta_2(0)(z_g^C-z)} \label{eq:dgc}\\
&\leq& -\beta \e^{\theta_2(0)(z_g^C-z)} \leq -\beta \nonumber\;.
\end{eqnarray}
This shows that $g_C$ solves \eqref{problem2}.
 \end{proof}
The following result will enable us to find $C$ in a way such that $\ubar{f}_C'(d)=\ubar{g}_C'(d)$.
\begin{lemma} \label{lemma:dfg}
Let $\beta<\zeta$. 
\begin{enumerate}
\item The function $(\xi_2(\beta),\xi_1(\beta)) \to !!R$, $C\mapsto \ubar{f}_C'(d)$ is strictly increasing and $\ubar{f}_{\xi_1(\beta)}'(d)=-\beta$.
\item The function $(\xi_2(\beta),\xi_1(\beta)) \to !!R$, $C\mapsto \ubar{g}_C'(d)$ is strictly decreasing and $\ubar{g}_{\xi_2(\beta)}'(d)=-\beta$.
\end{enumerate}
\end{lemma}
\begin{proof}
We start with i). By Proposition \ref{prop:zf} we have $z_f^{\xi_1(\beta)}=d$. Since $\ubar{f}_C'(z_f^C)=-\beta$ holds by construction, this directly shows $\ubar{f}_{\xi_1(\beta)}'(d)=-\beta$. Moreover, we have seen that $C\mapsto z_f^C$ is increasing and that for fixed $z\in[0,d]$, we can write $\ubar{f}_C'(d)$ as a function $k(z_f^C)$ by equation \eqref{eq:dfc}:
\begin{equation*}
k(z_f^C):= -\beta \frac{\tau(d,z_f^C)}{(\theta_1(0)+\theta_2(0))\e^{\theta_1(0)z_f^C}(\e^{\theta_1(u_0)z_f^C}-\e^{-\theta_2(u_0)z_f^C})}\;,
\end{equation*}
where $\tau(d,z_f^C)$ is defined as in \eqref{eq:tau}. Since
\begin{eqnarray*}
k'(z_f^C) &=& \frac{\beta}{(\theta_1(0)+\theta_2(0))\e^{2\theta_1(0)z_f^C}\bigl(\e^{\theta_1(u_0)z_f^C}-\e^{-\theta_2(u_0)z_f^C}\bigr)^2}\e^{\theta_1(0)z_f^C}\bigl(\theta_1(u_0)\e^{\theta_1(u_0)z_f^C}+\theta_2(u_0)\e^{-\theta_2(u_0)z_f^C}\bigr) \\
&& \hskip.5cm {} \cdot \bigl[\bigl(\e^{\theta_1(0)d}-\e^{\theta_1(0)z_f^C}\e^{-\theta_2(0)(d-z_f^C)} \bigr)\bigl(\theta_2(u_0)\e^{\theta_1(u_0)z_f^C}+\theta_1(u_0)\e^{-\theta_2(u_0)z_f^C}\bigr)  \\
&& \hskip1cm {} + \bigl(\theta_1(0)\e^{\theta_1(0)d}+\theta_2(0)\e^{\theta_1(0)z_f^C}\e^{-\theta_2(0)(d-z_f^C)} \bigr) \bigl(\e^{\theta_1(u_0)z_f^C}-\e^{-\theta_2(u_0)z_f^C} \bigr)\bigr] > 0\;,
\end{eqnarray*}
we have
\[
\frac{\id}{\id C} \ubar{f}_C'(d) = k'(z_f^C) \cdot \bigl(\frac{\id}{\id C} z_f^C \bigr) > 0\;.
\]
Thus, $C\mapsto \ubar{f}_C'(d)$ is increasing. \\
In order to show ii), we note that $z_g^{\xi_2(\beta)}=d$ by Proposition \ref{prop:zg} and therefore $g_{\xi_2(\beta)}'(d) = -\beta$ by construction of $z_g^{\xi_2(\beta)}$. By equation \eqref{eq:dgc}, we have
\begin{equation*}
\frac{\id}{\id z_g^C}\ubar{g}_C'(d) = \beta \frac{\theta_2(u_0)-\theta_2(0)}{\theta_1(0)+\theta_2(0)} (-\theta_1(0)\e^{-\theta_1(0)(z_g^C-d)}-\theta_2(0)\e^{\theta_2(0)(z_g^C-d)})-\beta\theta_2(0)\e^{\theta_2(0)(z_g^C-d)} < 0\
\end{equation*}
and therefore
\[
\frac{\id}{\id C} \ubar{g}_C'(d) = \bigl(\frac{\id}{\id z_g^C}\ubar{g}_C'(d)\bigr) \cdot \bigl(\frac{\id}{\id C} z_g^C \bigr) < 0\;,
\]
where we used that $C\mapsto z_g^C$ is increasing. This shows that $C\mapsto\ubar{g}_C'(d)$ is decreasing.
\end{proof}
Now we have $h=f_C\one_{[0,d]}+g_C\one_{(d,\infty)}$ as a candidate for the value function. Moreover, a candidate for an optimal strategy is given by the feedback strategy $(U^{*}(t)=u^{*}(\Delta_z^{U^{*}}(t)))_{t\geq 0}$, where $u^{*}(z) = u_0 \one_{[0,z_f]\cup[z_g,\infty)}(z)$, $z\geq 0$. In order to show that $h(z)$ is indeed the value function, we have to prove existence of the controlled process. We can follow the approach of Brinker \cite{LVB} and describe $(\Delta_z^{U^{*}},M_z^{U^{*}})$ as a solution $(Z,z+Y)$ to the reflected SDE
\begin{equation}
Z_t = z + Y_t + \int_0^t (u^{*}(Z_s)-\mu) \id s - \sigma W_t\;, \label{SDE}
\end{equation}
where $(Z,Y)$ solves a Skorohod problem. Then, the existence of a pathwise
unique strong solution of \eqref{SDE} follows from \cite[Theorem
4.1]{Rozkosz}, \cite[Corollary 4.3]{Semrau} and \cite[Theorem 1.3]{Situ}. In
particular, $(U^{*}(t))_{t\geq 0}$ is an admissible strategy.

We can now specify the value function in the case $\beta<\zeta$.
\begin{theorem} \label{thm:drawdownsensitive}
If $\beta<\zeta$, then there exists a unique $C\in(\xi_2(\beta),\xi_1(\beta))$, such that 
\begin{equation}
v(z) = \bar{f}_C(z)\one_{[0,z_f^C]}(z)+\ubar{f}_C(z)\one_{(z_f^C,d]}+\ubar{g}_C(z)\one_{(d,z_g^C]}(z)+\bar{g}_C(z)\one_{(z_g^C,\infty)}(z)\;, \qquad z\geq 0\;, \label{solutionbetalargerzeta}
\end{equation}
solves \eqref{HJB},
where $\ubar{f}_C$, $\bar{f}_C$, $\ubar{g}_C$ and $\bar{g}_C$ are as in
Propositions~{\rm\ref{prop:zf}} and {\rm\ref{prop:zg}}. By the Verification
Theorem~{\rm\ref{verification}} $v$ is the value function. 
An optimal strategy is given by $U^{*}(t)=u^{*}(\Delta_z^{U^{*}}(t))$, $t\geq 0$, with the optimiser 
 \[
u^{*}(z) = u_0 \one_{[0,z_f]\cup[z_g,\infty)}(z) = u_0\one_{\{v'(z)\geq-\beta\}}\;.
 \]
\end{theorem}
\begin{proof}
In view of Propositions \ref{prop:zf} and \ref{prop:zg}, it remains to show that there exists $C\in(\xi_2(\beta),\xi_1(\beta))$, such that $\ubar{f}_C'(d)=\ubar{g}_C'(d)$. By Lemma \ref{lemma:dfg}, the function $C\mapsto \ubar{f}_C'(d)-\ubar{g}_C'(d)$ is strictly increasing. Moreover, we have $\ubar{f}_{\xi_1(\beta)}'(d)=-\beta$, $\ubar{f}_C'(d) < -\beta$ for $C<\xi_1(\beta)$, $\ubar{g}_{\xi_2(\beta)}'(d)=-\beta$ and $\ubar{g}_C'(d)<-\beta$ for $C> \xi_2(\beta)$, leading to
\begin{equation*}
\ubar{f}_{\xi_2(\beta)}'(d)-\ubar{g}_{\xi_2(\beta)}'(d) < 0 \quad \text{and} \quad \ubar{f}_{\xi_1(\beta)}'(d)-\ubar{g}_{\xi_1(\beta)}'(d) > 0\;.
\end{equation*}
Thus, there exists a unique $C\in(\xi_2(\beta),\xi_1(\beta))$, such that  $\ubar{f}_C'(d)=\ubar{g}_C'(d)$. We can therefore conclude that the right-hand side of \eqref{solutionbetalargerzeta} is a solution to the HJB equation \eqref{HJB}. According to Theorem \ref{verification}, this solution has to be the value function.
\end{proof}
\begin{rem}
The proof of Theorem \ref{thm:drawdownsensitive} also shows that the value function cannot be of the form $$h(z)=\bar{f}_C(z)\one_{[0,z_f^C]}(z)+\ubar{f}_C(z)\one_{(z_f^C,d]}(z)+\bar{g}_C(z)\one_{(d,\infty)}(z)$$ or $$h(z)=\bar{f}_C(z)\one_{[0,d]}(z)+\ubar{g}_C(z)\one_{(d,z_g^C]}(z)+\bar{g}_C(z)\one_{(z_g^C,\infty)}(z)\;.$$ For the first case we need $\ubar{f}_C'(d)=\bar{g}_C'(d)=-\beta$ and in the second case $\bar{f}_C'(d)=\ubar{g}_C'(d)=-\beta$. We can see in the proof above, that this can only be fulfilled if $\xi_1(\beta)=\xi_2(\beta)$, which is equal to the case $\beta=\zeta$. Note that $C=\xi_1(\beta)=\xi_2(\beta)$ also implies $z_f^C=z_g^C=d$, which means that we seamlessly transition to the case $\beta\geq\zeta$.
\hfill $\blacksquare$
\end{rem}
We can interpret this result as follows: If dividends are weighted moderately,
no dividends are paid if the drawdown is close to the critical boundary in
order to get away from it, or leave the critical area, respectively,
as fast as possible. Close to the running maximium, the value of the dividends
exceed the danger of falling below the critical boundary. At a certain distance to
the critical boundary in the critical area, one will spend so much time in the
critical area such that paying dividends is favourable to leaving the critical
area faster.

%% file: NumericalExamples.tex
\section{Numerical Examples} \label{section:numerical}
In this section we illustrate our findings. Figures \ref{fig:2}, \ref{fig:3} and \ref{fig:1} show the value function for a various selection of parameters. The underlying strategy is illustrated by different shades: Dividends are paid at the maximal rate $u_0$ in the areas where the graph is grey and no dividends are paid in the black area. The black vertical line marks the critical drawdown level $d$.

\begin{figure}
\centering
\includegraphics[width=\textwidth]{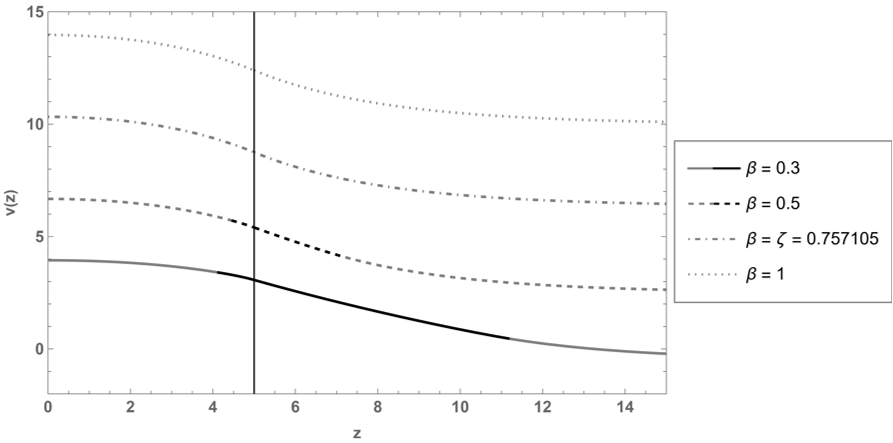}
\caption{Plot of $v(z)$ for parameters $\mu=3$, $\sigma=2$, $d=5$, $u_0=3$, $r=0.2$ and for various values of $\beta$.}
\label{fig:2}
\end{figure}
\begin{figure}
\centering
\includegraphics[width=\textwidth]{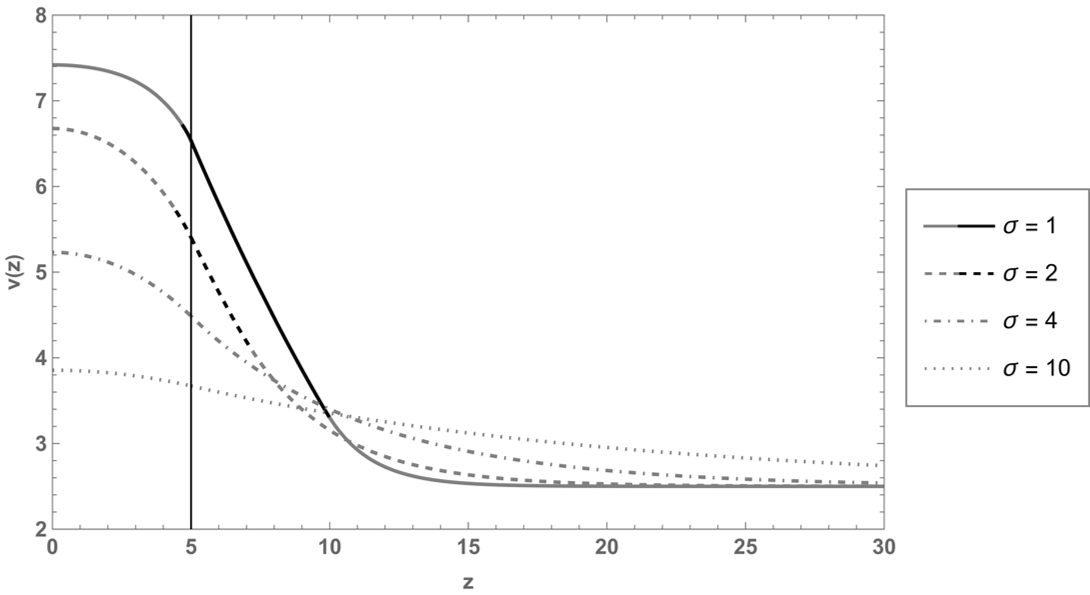}
\caption{Plot of $v(z)$ for parameters $\mu=3$, $u_0=3$, $d=5$, $\beta=0.5$, $r=0.2$ and for various values of $\sigma$.}
\label{fig:3}
\end{figure}
\begin{figure}
	\centering
	\includegraphics[width=\textwidth]{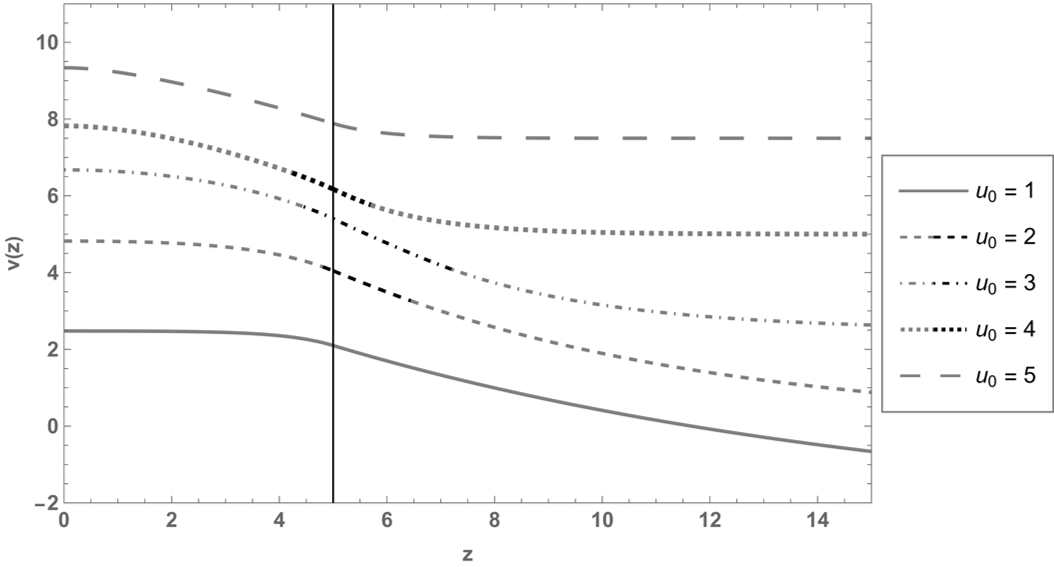}
	\caption{Plot of $v(z)$ for parameters $\mu=3$, $\sigma=2$, $d=5$, $\beta=0.5$, $r=0.2$ and for various values of $u_0$.}
	\label{fig:1}
\end{figure}
Unsurprisingly, we observe in Figure \ref{fig:2} and \ref{fig:1} that $v(z)$ is increasing in $\beta$ and $u_0$, respectively. Moreover, we can see in Figure \ref{fig:2} the convergence of $z_f$ and $z_g$ to $d$ as $\beta\uparrow\zeta$.

\noindent
In Figure \ref{fig:3} we can also see that $z_f$ and $z_g$ converge to $d$ as
$\sigma$ increases. This is because $\zeta$ is decreasing in $\sigma$ and
$\lim_{\sigma\to\infty} \zeta(\sigma) = 0$. For $\sigma$ large enough, we
therefore switch to the case $\beta>\zeta$. This is because the risk to reach
the critical area increases with $\sigma$ and cannot be prevented just by
not losing capital by paying dividends. Moreover, we observe that for larger
values of $\sigma$ the graph of $v(z)$ is flatter and converges slower towards
the limit value $(\beta u_0-1)/r$. This makes sense, since $\theta_2(u_0)$ is
decreasing in $\sigma$ and $v(z) \in ?O(\e^{-\theta_2(u_0)z})$. In particular,
we see in the examples considered in Figure \ref{fig:3}, that for large
initial drawdown $z$ the value $v(z)$ is larger if the volatility $\sigma$ is
higher. This can be interpreted in the following way: If the drawdown is
large, then we can return to smaller values through drift or through
fluctuations of the process. The size of the fluctuations depends on the size
of $\sigma$. If the volatility parameter $\sigma$ is large, the uncritical
area $[0,d)$ is ``easier'' to reach, even if we are further away from
it. Conversely, the area in which we can be pushed back to $[0,d)$ by
fluctuations is smaller when volatility is low. This explains why smaller
values of $\sigma$ lead to a larger decline in the value $v(z)$ near the
critical level $d$.

In Figure \ref{fig:1} we observe the symmetry of $\zeta$ in $u_0$, as mentioned in Remark \ref{rem:zeta}. That is, the size of the black area, where no dividends are paid, is maximal for $u_0=\mu$. We can explain this behaviour as follows: If $u_0\ll\mu$, then dividend payments only slightly reduce the drift. We therefore choose the constant strategy $U_t=u_0$, $t\geq 0$, as it involves little risk. On the other hand, if $u_0 \gg \mu$, then the value of dividend payments exceeds the costs of critical drawdown. We therefore accept that the drawdown process keeps growing and also choose the constant strategy $U_t=u_0$, $t\geq 0$.

\begin{figure}
\centering
\begin{tabular}{cc}
\includegraphics[width=0.68\textwidth]{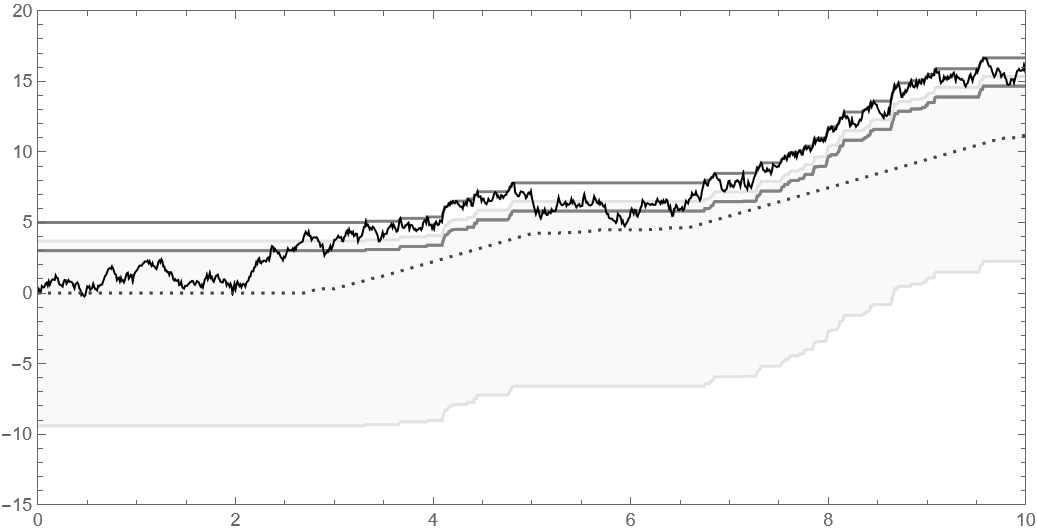} \\
\includegraphics[width=0.68\textwidth]{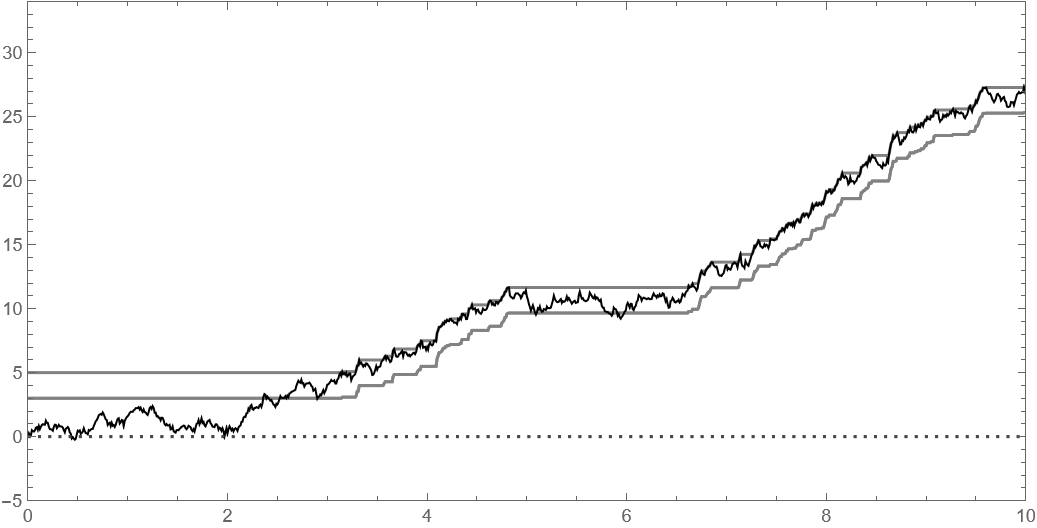}\\
\includegraphics[width=0.68\textwidth]{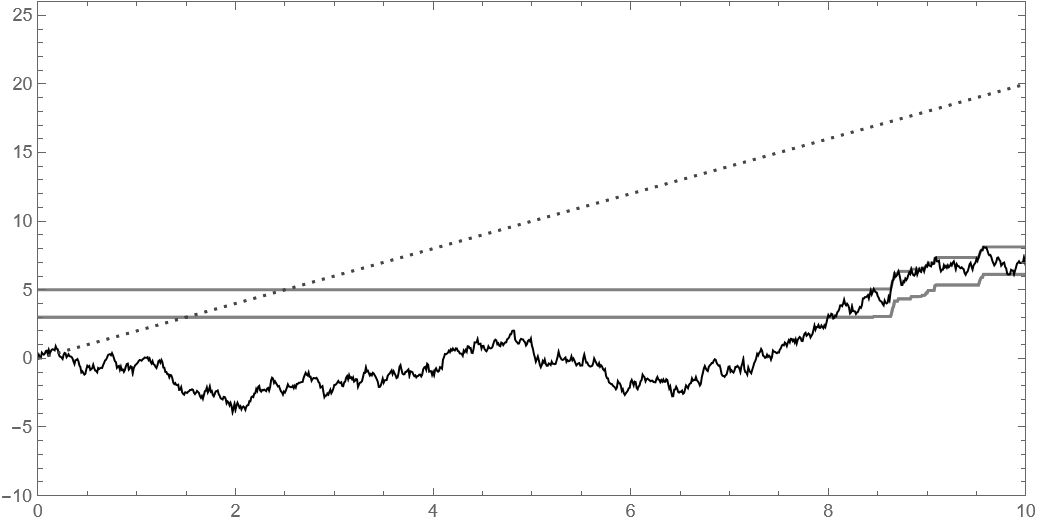}
\end{tabular}
\caption{Path simulations of the surplus process $X^{U}$ (black), the accumulated dividend payments $(D_t^U=\int_0^t U_s \id s)_{t\geq 0}$ (dotted) and the running maximum $M^U$ (grey) under the strategies $U=U^{*}$, $U=0$ and $U=u_0$ (from top to bottom) for $\mu>u_0=2$.}
\label{fig:simulation1}
\end{figure}

\begin{figure}
\centering
\begin{tabular}{c}
\includegraphics[width=0.68\textwidth]{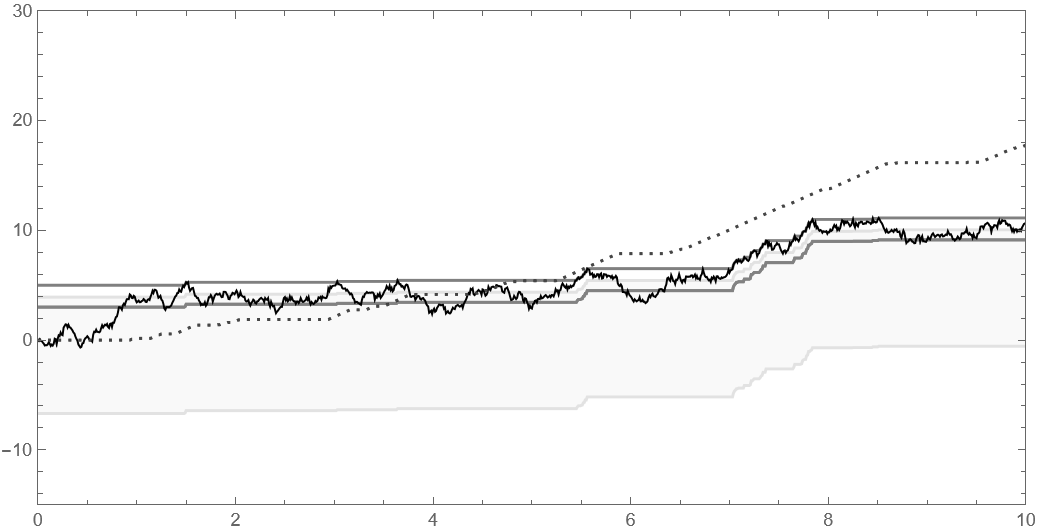} \\
\includegraphics[width=0.68\textwidth]{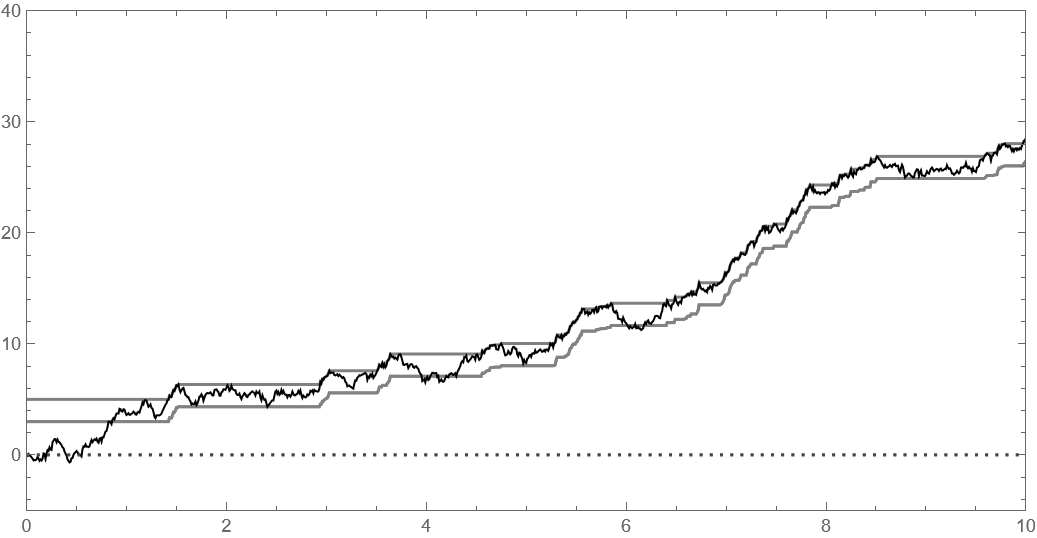}\\
\includegraphics[width=0.68\textwidth]{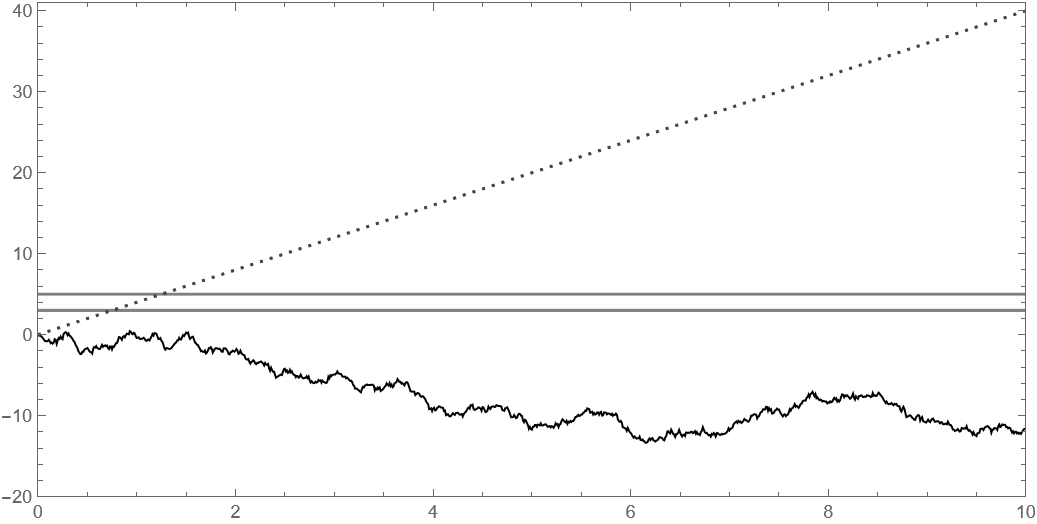}
\end{tabular}
\caption{Path simulations of the surplus process $X^{U}$ (black), the accumulated dividend payments $(D_t^U=\int_0^t U_s \id s)_{t\geq 0}$ (dotted) and the running maximum $M^U$ (grey) under the strategies $U=U^{*}$, $U=0$ and $U=u_0$ (from top to bottom) for $\mu<u_0=4$. }
\label{fig:simulation2}
\end{figure}
The behaviour of the surplus process is illustrated in Figures
\ref{fig:simulation1}, \ref{fig:simulation2} and \ref{fig:extreme} for chosen
parameters $\mu=3$, $\sigma=2$, $d=2$, $\beta=0.15$ and $r=0.2$. Figure
\ref{fig:extreme} illustrates a scenario for large initial drawdown and
$\mu<u_0$ under the optimal strategy $U^{*}$. In Figures \ref{fig:simulation1}
and \ref{fig:simulation2}, the same realisation of the underlying risk process
was used to simulate the surplus under the strategies $U=U^*$, $U=0$ and
$U=u_0$.  Here, the two grey lines mark the area between the running maximum and the critical drawdown level $d$. The
light grey area marks the interval $[z_f,z_g]$, where no dividends are paid. As
expected, we observe that drawdown control is best for the strategy $U=0$ and
worst for $U=u_0$. Under the optimal strategy $U^{*}$, we find that the two
objectives, dividend payments and non-critical drawdowns, are reconciled: the
process aims to stay in the uppermost area where dividends are paid and the
drawdown is non-critical. If the process reaches the critical area it does
not get far from the boundary $d$.

\begin{figure}
\centering
\includegraphics[width=0.68\textwidth]{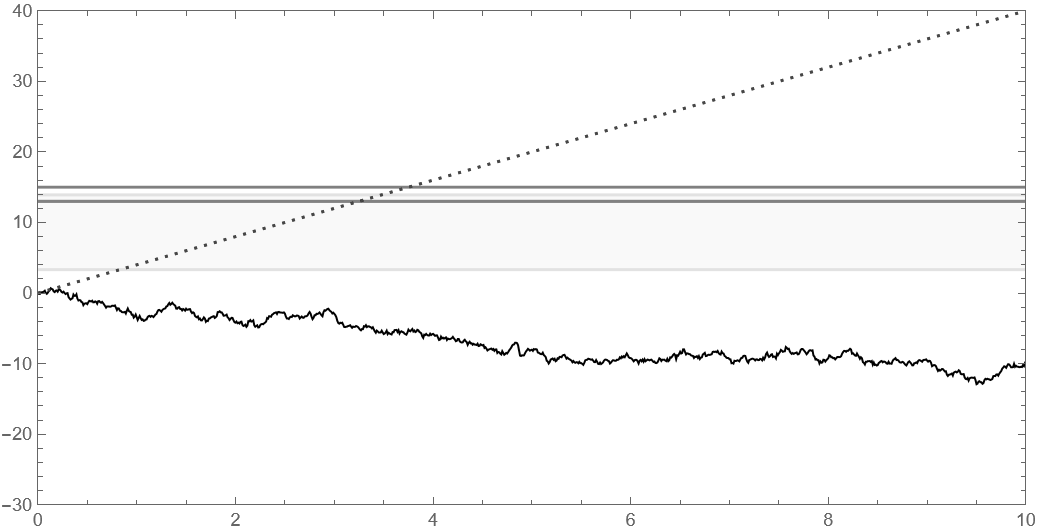}
\caption{Path simulations of the surplus process $X^{U}$ (black), the accumulated dividend payments $(D_t^U=\int_0^t U_s \id s)_{t\geq 0}$ (dotted) and the running maximum $M^U$ (grey) under the strategy $U=U^{*}$ for $\mu<u_0=4$.}
\label{fig:extreme}
\end{figure}

%% file: Conclusion.tex
\section{Conclusion}
We considered a diffusion risk model where dividends are paid at rate $U_t\in[0,u_0]$. The goal was to maximise dividend payments while keeping drawdowns below a critical level $d$. We linked the optimisation problem to a Hamilton--Jacobi--Bellman equation. We showed that a solution exists and that the optimal dividend rate is either $0$ or $u_0$. We determined a critical value $\zeta$ and showed that the solution to the Hamilton--Jacobi--Bellman equation can only take two different forms, depending on the cases $\beta\geq\zeta$ and $\beta<\zeta$, respectively. In the case $\beta\geq\zeta$, we constantly pay dividends at the maximal rate $u_0$. If $\beta<\zeta$, then no dividends are paid if the drawdown is close to the critical boundary $d$. Finally, we gave some numerical examples to illustrate our findings.

The results of our analysis cause us to critically question our model. We have seen that it is optimal to pay dividends at the maximal rate $u_0$ if the drawdown is large. In the case $\mu< u_0$, where the drift is smaller than the dividend rate, this would lead to unbounded growth of the drawdown. In reality, this is unacceptable as we cannot accumulate unlimited debt. One option would be to stop at the time of ruin, as done in the original dividend problem (compare \cite{Schmidli}).  However, bankruptcy rarely occurs in reality. Another way would therefore be to choose a penalty function that takes the size of the drawdown into account, as described by Brinker \cite{LVB}. The results on optimal dividends with penalty payments of Vierk\"otter \cite{Vierkoetter} could serve as an inspiration here.